\documentclass[a4study,twoside,11pt]{article}
\usepackage{amssymb}
\usepackage{amsmath}
\usepackage{graphicx}
\usepackage{amsthm}
\usepackage{cite}
\usepackage{array}

\usepackage{amsmath,latexsym,amssymb,amsfonts,amsbsy,amsthm}
\usepackage{graphicx}
\usepackage{subfigure}
\usepackage[a4paper]{geometry}
\usepackage{epstopdf}
\usepackage{diagbox}
\usepackage{enumerate}
\usepackage{color}

\usepackage{comment}
\usepackage{indentfirst}

\newfont{\bb}{msbm10}

\newtheorem{theorem}{Theorem}[section]
\newtheorem{definition}{Definition}[section]
\newtheorem{lemma}{Lemma}[section]

\newtheorem{remark}{Remark}[section]

\newtheorem{prop}{Proposition}[section]

\usepackage{verbatim}
\usepackage{algorithm}
\usepackage{algorithmicx}
\usepackage{algpseudocode}
\usepackage{amsmath}
\usepackage{graphics}
\usepackage{epsfig}
\usepackage{arydshln}

\usepackage{mathrsfs}
\usepackage{graphicx}
\usepackage{subfigure}
\usepackage{caption}
\usepackage{multirow}
\usepackage{threeparttable}
\numberwithin{equation}{section}

\usepackage{float}
\usepackage{footnote}
\makesavenoteenv{table}
\usepackage{booktabs}

\usepackage{authblk}

\usepackage{bm}
\usepackage{hyperref} 

\usepackage{newtxmath}

\title{Greedy randomized Bregman-Kaczmarz method for constrained nonlinear systems of equations} 
 
\author{
	A-Qin Xiao\\
	School of Mathematical Sciences, Tongji University, \\
	Shanghai, 200092, PR China. \\
	Email:xiaoaqin@tongji.edu.cn \\ 
        and \\
	Jun-Feng Yin\thanks{Corresponding author.}\\
	School of Mathematical Sciences, Tongji University, \\ 
 Key Laboratory of Intelligent Computing and Applications,
 Ministry of Education,\\
	Shanghai, 200092, PR China.\\
	Email:yinjf@tongji.edu.cn
}

\begin{document}
 \date{ }
\maketitle

\begin{abstract}
A greedy randomized nonlinear Bregman-Kaczmarz method by sampling the working index with residual information is developed for the solution of the constrained nonlinear system of equations. 
Theoretical analyses prove the convergence of the greedy randomized nonlinear Bregman-Kaczmarz method and its relaxed version.
Numerical experiments verify the effectiveness of the proposed method, which converges faster than the existing nonlinear Bregman-Kaczmarz methods.
\end{abstract}


\section{Introduction} \label{sec:intro_grnbk}	
Consider the constrained nonlinear system of equations
\begin{equation}
 \label{eqn:problem}
\hat{x}= \mathop{\mathrm{argmin}}\limits_{x\in C} \varphi(x), \quad \text{s.t. } F(x) = 0,
\end{equation}
where $\varphi(x)$ is a convex function, $F:C \to\mathbb{R}^n $ is a nonlinear differentiable function and $C\subset \mathbb{R}^d$ is a nonempty closed convex set. This nonlinear problem often arises in many scientific and engineering computations, for instance, optimization problems \cite{2011Y}, phase retrieval \cite{2020FS} and deep learning \cite{2022LZB}. 

The nonlinear Kaczmarz method is an effective iterative method for solving the nonlinear system of equations \cite{2022JYN}. 
A randomized nonlinear Kaczmarz method was proposed by choosing the index with probability \cite{2022WLB}. The randomized nonlinear Kaczmarz method was further accelerated by selecting the index corresponding the maximum residual of partial \cite{2023ZBLW} or complete \cite{2023ZWZ} nonlinear equations. 
For more related studies on the nonlinear Kaczmarz methods, we refer the reader to \cite{2021GHT,2022YLG,2024XYb}. 
Recently, for solving the constrained nonlinear problem \eqref{eqn:problem}, by combining the nonlinear Kaczmarz method with general Bregman projections, a nonlinear Bregman-Kaczmarz method was proposed and its expected convergence was given \cite{2024GLW}.  

In this paper, to accelerate the convergence of the nonlinear Bregman-Kaczmarz method, by choosing the working index according to the probability with residual information, a greedy randomized nonlinear Bregman-Kaczmarz method is presented.
The convergence theory of the proposed method and its relaxed version is established under the classical tangential cone conditions. 
Numerical experiments on different nonlinear systems with various constraints show the efficiency of the greedy randomized nonlinear Bregman-Kaczmarz method and the remarkable acceleration in convergence to the solution.  

The rest of this paper is organized as follows. In Section \ref{secprelim_grnbk}, basic properties on convex optimization are introduced.
The greedy randomized nonlinear Bregman-Kaczmarz method and its relaxed version are proposed in Section \ref{secmrnbrgk_grnbk}. The convergence theory of the proposed method and its relaxed version is established in Section \ref{secconvanal_grnbk}. 
Numerical experiments are provided to show the effectiveness of the proposed methods in Section \ref{secnumer_grnbk}.
Finally, some remarks and conclusions are drawn in Section \ref{secconclu_grnbk}.

\section{Preliminaries}\label{secprelim_grnbk}
	\label{sec:cvx-ana-basic-assumption}
Let $\varphi:\mathbb{R}^d\to\overline{\mathbb{R}} :=\mathbb{R}\cup\{+\infty\}$ be a convex function with 
$\text{dom }\varphi = \{x\in\mathbb{R}^d: \varphi(x)<\infty\} \neq\emptyset. $
Assuming that $\varphi$ is lower semicontinuous with
$\varphi(x) \leq \lim\inf\limits_{y\to x} \varphi(y)$  for all $x\in\mathbb{R}^d$, and is {supercoercive}, that is
$ \lim\limits_{\|x\|\to\infty} \frac{\varphi(x)}{\|x\|} = +\infty. $
\begin{definition}\rm
The subdifferential at $x\in\text{dom }\varphi$ is defined as 
$$
\partial\varphi(x)=\big\{x^*\in\mathbb{R}^d:\varphi(x) + \langle x^*,y-x\rangle \leq \varphi(y), \quad y\in \text{dom }\varphi\big\},
$$
where the vector $x^*\in\partial\varphi(x)$ represents a {subgradient} of $\varphi$ at $x$.
and $\text{dom }\partial\varphi$ is set of all points $x$ with $\partial\varphi(x)\neq\emptyset$ is denoted by $\text{dom }\partial\varphi$. 
\end{definition}
Note that the relative interior of $\text{dom }\varphi$ is a convex set, while $\text{dom }\partial\varphi$ may not be convex.
In general, convexity of $\varphi$ guarantees the inclusions 
$\text{ri }\text{dom }\varphi \subset \text{dom }\partial\varphi\subset \text{dom }\varphi$. 
Furthermore, assume that $\varphi$ is {essentially strictly convex} \cite{2024GLW}, meaning that strictly convex on $\text{ri }\text{dom }\varphi$. 
\begin{definition}\rm
The convex conjugate of $\varphi$ is defined as
$$
\varphi^*(x^*) = \sup_{x\in\mathbb{R}^d} \langle x^*,x\rangle - \varphi(x), \qquad x^*\in\mathbb{R}^d, 
$$
where the function $\varphi^*$ is convex and lower semicontinuous.
\end{definition}
In addition, the essential strict convexity and supercoercivity indicate that $\text{dom }\varphi^*=\mathbb{R}^d$ and $\varphi^*$ is differentiable, since $\varphi$ is essentially strictly convex and supercoercive.
\begin{definition} \rm
The Bregman distance \cite{2019FD} between $x$ and $y$ with respect to $\varphi$ and a subgradient $x^*\in\partial \varphi(x)$ is defined as
$$
D_\varphi^{x^*}(x,y) = \varphi(y) - \varphi(x) - \langle x^*,y-x\rangle, \qquad x,y\in\text{dom }\varphi.
$$
\end{definition}
Using Fenchel's equality $\varphi^*(x^*)=\langle x^*,x\rangle - \varphi(x)$ for $x^*\in\partial\varphi(x)$,
the Bregman distance with the conjugate function can be rewritten as
\begin{equation} 
	\label{eqn:BregDist_with_conjugate}
D_\varphi^{x^*}(x,y) = \varphi^*(x^*) - \langle x^*,y\rangle + \varphi(y).
  \end{equation} 	
If $\varphi$ is differentiable at $x$, then the subdifferential $\partial\varphi(x)$ contains the single element $\nabla\varphi(x)$ and it yields that
$$
 D_\varphi(x,y) = D^{\nabla\varphi(x)}_\varphi(x,y) = \varphi(y) - \varphi(x) - \langle \nabla\varphi(x),y-x\rangle. 
$$
\begin{definition}\rm
The function $\varphi$ is called $\sigma$-{strongly} convex with respect to a norm $\|\cdot\|$ for some $\sigma>0$, if for all $x,y\in\text{dom } \partial\varphi$ it satisfies that $\frac{\sigma}{2}\|x-y\|^2\leq D_\varphi^{x^*}(x,y)$. 
 \end{definition}

\begin{definition}\rm\textbf{({Bregman projection}\cite{2019FD}).}\label{defbrgproj}
 Let $E \subset \mathbb{R}^d$ be a nonempty convex set, $E \cap \text{dom }\varphi\neq\emptyset$, $x\in\text{dom }\partial\varphi$ and $x^*\in\partial\varphi(x)$. The Bregman projection of $x$ onto~$E$ with respect to $\varphi$ and $x^*$ is the point $\Pi_{\varphi, E}^{x^*}(x)\in E \cap \text{dom }\varphi$ such that 
	\begin{equation*}
			D_\varphi^{x^*}\big(x,\Pi_{\varphi, E}^{x^*}(x)\big) = \min_{y\in E} D_\varphi^{x^*}(x,y). 
	\end{equation*}
\end{definition}
Note that the existence and uniqueness of the Bregman projection is guaranteed if $E\cap\text{dom }\varphi\neq\emptyset$ by the \cite{2024GLW}[Assumption 1], since the fact that the function $y\mapsto D_\varphi^{x^*}\big(x, y\big)$ is lower bounded by zero, coercive, lower semicontinuous and strictly convex. 
Specially, for the standard quadratic $\varphi=\frac{1}{2}\|\cdot\|_2^2$, the Bregman projection is just the orthogonal projection.
Note that if $E\cap\text{dom }\varphi=\emptyset$, then it holds that $D_\varphi^{x^*}(x,y)=+\infty$ for all $y\in E$. 

Considering the Bregman projections onto hyperplanes 
	\[ H(\alpha,\beta):= \{x\in\mathbb{R}^d: \langle \alpha,x\rangle = \beta\}, \qquad \alpha\in\mathbb{R}^d, \ \beta\in\mathbb{R}, \]
	and halfspaces
\begin{equation*}
		H^{\leq(\geq)}(\alpha,\beta):= \{x\in\mathbb{R}^d: \langle \alpha,x\rangle \leq(\geq) \beta\}, \qquad \alpha\in\mathbb{R}^d, \ \beta\in\mathbb{R}.
\end{equation*}
	
The following proposition shows that the Bregman projection onto a hyperplane can be computed by solving a one-dimensional dual problem under a qualification constraint. Here, this dual problem can be formulated under slightly more general assumptions than previous versions, for example, neither assume smoothness of $\varphi$ \cite{2003BC,2008DT} nor strong convexity of $\varphi$ \cite{2014LSW}.
	
\begin{prop} \rm
\label{prop:t_min_problem_abstract_hyperplane}
 Let $\alpha\in\mathbb{R}^d\setminus\{0\}$ and $\beta\in\mathbb{R}$ such that 
\[ H(\alpha,\beta)\cap \text{ri }\text{dom } \varphi \neq\emptyset. \]
Then, for all $x\in\text{dom }\partial\varphi$ and $x^*\in\partial\varphi(x)$, the Bregman projection $\Pi_{\varphi, H(\alpha,\beta)}^{x^*}(x)$ exists and is unique, which is given by 
$$
x_+:=\Pi_{\varphi, H(\alpha,\beta)}^{x^*}(x) = \nabla\varphi^*(x_+^*), 
$$
where $x_+^* = x^* - \hat t\alpha \in \partial\varphi(x_+)$ and $\hat t$ is a solution to 
\begin{equation}
		\label{eqn:t_min_problem_abstract_hyperplane}
	 \min_{t\in\mathbb{R}} \varphi^*(x^*-t\alpha) + \beta t.
  \end{equation}
\end{prop}

\section{The greedy randomized nonlinear Bregman-Kaczmarz method} \label{secmrnbrgk_grnbk}
In this section, the nonlinear Bregman-Kaczmarz method is firstly reviewed and then a greedy randomized nonlinear Bregman-Kaczmarz method as well as its relaxed version are presented for solving problem \eqref{eqn:problem}.

Given an appropriate convex function $\varphi\colon\mathbb{R}^d\to \mathbb{R} \cup \{+\infty\}$ with 
$\overline{\text{dom }\partial\varphi}= C$.
The nonlinear Bregman-Kaczmarz method calculates the {Bregman projection} with respect to $\varphi$ onto the solution set of  the local linearization of a component of equation $F_{i_k}(x)=0$ at the current iterate $x_k$, where the index $i_k$ is uniform randomly chosen from $i_k\in 
\{1,\ldots, n\}$.
Then, the problem \eqref{eqn:problem} can be stated as the following constrained optimization problem
\begin{equation} \label{eq:ourmethodintro}
	x_{k+1} =  \mathop{\mathrm{argmin}}_{x\in\mathbb{R}^d} D_\varphi^{x_k^*}(x_k,x), \qquad \text{s.t. } x\in H_k,
	\end{equation}
with
\begin{equation*}
H_k:= \{x\in\mathbb{R}^d: F_{i_k}(x_k) + \langle \nabla F_{i_k}(x_k), x-x_k\rangle = 0\}= H(\nabla F_{i_k}(x_k), \beta_k),
\end{equation*} 
where
	\begin{equation}
		\label{eqn:beta_k}
		\beta_k = \langle \nabla F_{i_k}(x_k),x_k\rangle - F_{i_k}(x_k).
	\end{equation} 
Since the iterate $x_{k+1}$ is obtained by taking the Bregman projection of $x_k$ onto the set $H_k$ and using Proposition~\ref{prop:t_min_problem_abstract_hyperplane}, which is possible if
	\begin{equation}
		\label{eqn:Condition_hyperplane_nonempty_intersection}
		H_k \cap \text{dom }\partial\varphi\neq\emptyset.
	\end{equation}
Then the nonlinear Bregman-Kaczmarz method updates $x_k$  by $x_{k+1}^* = x_k^*-t_{k,\varphi}\nabla F_{i_k}(x_k)$ and $x_{k+1}=\nabla\varphi^*(x_{k+1}^*)$ with
	\begin{equation}
		\label{eqn:BregProj_stepsize}
		t_{k,\varphi} \in \mathop{\mathrm{argmin}}_{t\in\mathbb{R}} \varphi^*(x_k^*-t\nabla F_{i_k}(x_k)) + \beta_k t,
	\end{equation}
where the working index $i_k$ is sampled randomly from $\{1,2,\cdots,n\}$.

Note that the Bregman projection $x_{k+1}$ is unique,  but $t_{k,\varphi}$ might not be unique. If~\eqref{eqn:Condition_hyperplane_nonempty_intersection} is not satisfied, an update is defined by setting $x_{k+1} = \nabla\varphi^*(x_k^*-t_{k,\sigma}\nabla F_{i_k}(x_k))$
with 
	\begin{equation}
		\label{eqn:mSPS_like_stepsize}
		t_{k,\sigma} =  \sigma\frac{F_{i_k}(x_k)}{\|\nabla F_{i_k}(x_k)\|_*^2},
	 \end{equation}
where $\|\cdot\|_*$ is a norm and the constant $\sigma>0$.
The authors referred to the resulting update as the {relaxed projection} and then proposed a relaxed version of the nonlinear Bregman-Kaczmarz method. 

The convergence properties of the nonlinear Bregman-Kaczmarz method and its relaxed version in a classical setting of nonlinearity are restated in Theorem \ref{thm:convtccNBK}.

\begin{theorem}\rm 
	\label{thm:convtccNBK}
Let \cite[Assumption 1]{2024GLW} hold true and $\varphi$ be $\sigma$-strongly convex. Let each  function $F_i$ satisfies the local tangential cone condition with some $\eta>0$ and $\hat x\in S$ and let $x_0\in B_{r,\varphi}(\hat x)$. Moreover, assume that the Jacobian $F'(x)$ has full column rank for all $x\in B_{r,\varphi}(\hat x)$ and $p_{\min}=\min\limits_{i=1,...,n} p_i>0$. Set
	\begin{equation*}
		\kappa_{\min}  := \min_{x\in B_{r,\varphi}(\hat x)} \min_{\|y\|_2=1} \frac{\|F'(x)\|_F}{ \|F'(x)y\|_2},
		\end{equation*}
where $\|\cdot\|_F$ is the Frobenius norm.		
		\begin{enumerate}
			\item[(i)] If $\eta<\frac12$, then the iterates $x_k$ generated by the relaxed nonlinear Bregman-Kaczmarz method fulfill that 
			\begin{equation} 
				\label{eqn:tcc_linear}
				\frac{\sigma}{2}\mathbb{E}\big[\|x_k-\hat x\|_2^2] 
				\leq \mathbb{E}\big[D_\varphi(x_{k},\hat x)\big] 
				\leq \Big( 1- \frac{ \sigma\big(\frac12-\eta\big)p_{\min}}{M(1+\eta)^2\kappa_{\min}^2} \Big)^k \mathbb{E}\big[D_\varphi(x_0,\hat x)\big].
			\end{equation}
			\item[(ii)] Let $\varphi$ be additionally $M$-smooth and $\eta < \frac{\sigma}{2M}$. Assume that $H_k\cap\text{dom }\varphi\neq\emptyset$ and $x_k$ are the iterates generated by the nonlinear Bregman-Kaczmarz method. Then it holds that 
			\begin{equation} 
				\label{eqn:tcc_linear_Alg2}
				\frac{\sigma}{2}\mathbb{E}\big[\|x_k-\hat x\|_2^2] 
				\leq \mathbb{E}\big[D_\varphi(x_{k},\hat x)\big] 
				\leq \Big( 1- \frac{ \sigma\big(\frac12-\eta\frac{M}{\sigma}\big)p_{\min}}{M(1+\eta)^2\kappa_{\min}^2} \Big)^k \mathbb{E}\big[D_\varphi(x_0,\hat x)\big].
			\end{equation}
		\end{enumerate}
\end{theorem}
 
It is known that when the number of iteration steps is large, the iteration process of the nonlinear Bregman-Kaczmarz method and its relaxed version will traverse all the rows of $F(x_k)$, leading to a slowly convergence.
 
In this paper, to accelerate the convergence of the nonlinear Bregman-Kaczmarz method, a greedy strategy is adopted to choose the working row with large entry of the residual vector are grasped at each iteration. Moreover, to save the storage and computational cost, only the residual and its norm are calculated. 
That is, the working row is selected with the probability $p_{i_k}={|r_k^{(i)}|^2}/{\|r_k\|^2}$,
where $r_k= -F(x_k)$ is the residual vector with respect to $x_k$ and $r_k^{(i)}$ is its $i$-th component. For more studies on the selection strategies with residual information, we refer the readers to see \cite{2021HM,2021BW,2023XYZ,2024ZLT}.

Based on this criterion, a greedy randomized nonlinear Bregman-Kaczmarz method is proposed in Algorithm \ref{alg:GRNBK}.

 \begin{algorithm}[!htbp] 
 \caption{Greedy randomized nonlinear Bregman-Kaczmarz method}
 \label{alg:GRNBK}
\begin{algorithmic}[1]
	\Require $x_{0}^*\in \mathbb{R}^d, x_0 = \nabla \varphi^*(x_0^*), r_0= -F(x_0)$ and $\sigma>0$. 
        \Ensure $x_\ell$.
		\For{$k=0,1,\cdots,\ell-1$}
                \State select an index $i_k$ with probability $ p_{i_k}=\frac{|r_k^{(i)}|^2}{\|r_k\|^2}$\			
			\If{$F_{i_k}(x_k)\neq 0$ and $\nabla F_{i_k}(x_k)\neq 0$} \\
			\Comment{otherwise, the component equation is solved already, or $H_k=\emptyset$}
			\State set $\beta_k = \langle \nabla F_{i_k}(x_k),x_k\rangle - F_{i_k}(x_k)$
			\If{ $H_k \cap \text{dom }\partial\varphi\neq\emptyset$ }
			\State  \vspace{-0.3cm}
			$$ \displaystyle      \mbox{ Find $t_k$:} \quad
			t_k \in \mathop{\mathrm{argmin}}_{t\in\mathbb{R}} \, \varphi^*(x_k^*-t\nabla F_{i_k}(x_k)) + t\beta_k 	$$ 
			\vspace{-0.5cm}
			\Else{ set $t_k = \sigma \frac{F_{i_k}(x_k)}{\|\nabla F_{i_k}(x_k)\|_*^2}$ }
			\EndIf 
			\State update $x_{k+1}^*= x_k^* - t_k \nabla F_{i_k}(x_k) $
			\State update $x_{k+1} = \nabla \varphi^*(x_{k+1}^*)$ 
			\EndIf
     \State Compute residual $r_{k+1}= -F(x_{k+1})$
			\EndFor
		\end{algorithmic} 
	\end{algorithm}	

 As an alternative method, the relaxed method which always chooses the stepsize $t_{k,\sigma}$ from \eqref{eqn:mSPS_like_stepsize} is also considered, see Algorithm \ref{alg:rGRNBK}.
 
\begin{algorithm}[!htbp] 
 \caption{Relaxed greedy randomized nonlinear Bregman-Kaczmarz method}
 \label{alg:rGRNBK}
\begin{algorithmic}[1]
	\Require $x_{0}^*\in \mathbb{R}^d, x_0 = \nabla \varphi^*(x_0^*), r_0= -F(x_0)$ and $\sigma>0$. 
        \Ensure $x_\ell$. 
			\For{$k=0,1,\cdots,\ell-1$}
                \State select an index $i_k$ with probability $ p_{i_k}=\frac{|r_k^{(i)}|^2}{\|r_k\|^2}$ 	 
			\If{ $H_k \cap \text{dom }\partial\varphi\neq\emptyset$ }
			\State  set $t_k = \sigma \frac{F_{i_k}(x_k)}{\|\nabla F_{i_k}(x_k)\|_*^2}$  
			 \State update $x_{k+1}^*= x_k^* - t_k \nabla F_{i_k}(x_k) $
			\State update $x_{k+1} = \nabla \varphi^*(x_{k+1}^*)$ 
			\EndIf
     \State Compute residual $r_{k+1}= -F(x_{k+1})$
			\EndFor
		\end{algorithmic} 
	\end{algorithm}

Note that when the function $\varphi(x) = \frac12\|x\|_2^2$, the Algorithm \ref{alg:rGRNBK} method covers the randomized nonlinear Kaczmarz method \cite{2023ZWZ}.

\section{ Convergence analysis}\label{secconvanal_grnbk}
In this section, the convergence properties of the greedy randomized nonlinear Bregman-Kaczmarz method and its relaxed version under a typical settings of the nonlinearity are discussed.
Firstly, some lemmas and definitions are introduced, which are crucial for the following convergence analysis of the proposed method.

\begin{lemma}\rm
	\label{lem:BasicsConvexAnalysis_strongly_convex}
		If $\varphi\colon\mathbb{R}^d\to\mathbb{R}$ is proper, convex and lower semicontinuous, then the following statements are equivalent:  
		\begin{enumerate}[(i)]
			\item $\varphi$ is $\sigma$-strongly convex with respect to $\|\cdot\|$. 
			\item $ \langle x^*-y^*, x-y\rangle \geq \sigma\|x-y\|^2$, $x,y\in\mathbb{R}^d$ and $x^*\in\partial\varphi(x)$, $y^*\in\partial\varphi(y)$.
			\item The function $\varphi^*$ is $\tfrac{1}{\sigma}$-smooth with respect to $\|\cdot\|_*$.
		\end{enumerate}
	\end{lemma}
\begin{lemma}\rm	\label{lem:BasicsConvexAnalysis_convex_Lsmooth}
		If $\varphi\colon\mathbb{R}^d\to\mathbb{R}$ is convex and lower semicontinuous, then the following statements are equivalent: 
		\begin{enumerate}[(i)]
			\item $\varphi$ is $M$-smooth with respect to a norm $\|\cdot\|$, 
			\item $\varphi(y) \leq \varphi(x) + \langle \nabla\varphi(x),y-x\rangle + \frac{M}{2} \|x-y\|^2$ for all $x,y\in\mathbb{R}^d$,
			\item $\langle \nabla\varphi(y) - \nabla\varphi(x),y-x\rangle \leq M\|x-y\|^2$ for all $x,y\in\mathbb{R}^d$.
		\end{enumerate}
\end{lemma}

\begin{lemma}\rm
	\label{lem:quot_diff_estimate}
		Let $a_1,...,a_n\geq 0$ and $b_1,...,b_n>0$. Then it holds that 
		\[ \sum_{i=1}^n \frac{a_i}{b_i} \geq \frac{\sum_{i=1}^n a_i}{\sum_{i=1}^n b_i}. \] 
\end{lemma}

\begin{definition}\rm(\cite{2022WLB})
   A differentiable function $F: \mathbb{R}^d\to\mathbb{R}^n$ fulfills the 
  {local tangential cone condition}  with constant $0<\eta<1$, if for all $x,y\in \mathbb{R}^d$ it holds that
\begin{equation}
  \label{eqn:TCC1}
  |F(x)+\langle \nabla F(x),y-x\rangle -F(y)| \leq \eta |F(x)-F(y)|
  \end{equation}
  and
  \begin{equation}\label{eqn:TCC2}
  |\langle \nabla F_{i}(x), x - y\rangle| 		
		 \leq {(1+\eta)} |F_{i}(x)-F_{i}(x)|.
  \end{equation}
\end{definition}

\begin{lemma}\rm 
\label{lem:Dphi_descent_estimate_tcc}
Let $\varphi$ be a $\sigma$-strongly convex with respect to a norm~$\|\cdot\|$, there exist $\hat x\in S$, constants $\eta\in (0,1)$ and $r>0$ such that each function $F_i$ satisfies the local tangential cone condition with respect to $\eta$ on 
$B_{r,\varphi}(\hat x) := \big\{x\in C: D_\varphi^{x^*}(x,\hat x) \leq r \quad \text{for all } x^*\in\partial\varphi(x) \big\}$. Then, the iterates of Algorithm~\ref{alg:GRNBK} hold that
 \[ D^{x_{k+1}^*}_\varphi(x_{k+1},\hat x) \leq D^{x_k^*}_\varphi(x_k,\hat x) - \tau \frac{\big(F_{i_k}(x_k)\big)^2}{\|\nabla F_{i_k}(x_k)\|_*^2}, \] 
		if one of the following conditions is fulfilled:
		\begin{enumerate}[(i)]
			\item $t_k=t_{k,\sigma}$, $\eta<\frac12$ and $\tau = \sigma\big(\frac{1}{2}-\eta\big)$,
			\item $t_k=t_{k,\varphi}$, $\varphi$ is additionally $M$-smooth with respect to $\|\cdot\|$, $\eta < \frac{\sigma}{2M}$ and $\tau = \sigma\big(\frac{1}{2}-\eta \frac{M}{\sigma}\big).$
		\end{enumerate}
	In particular, if $x_0\in B_{r,\varphi}(\hat x)$, then in both cases, it has that $x_k\in B_{r,\varphi}(\hat x)$ for all $k$.
\end{lemma}

The Lemma \ref{lem:Dphi_descent_estimate_tcc} provides an important property for the sequence $\{x_k\}_{k=0}^{\infty}$ generated by the iterative scheme of greedy randomized nonlinear Bregman-Kaczmarz method, which is useful for the later convergence analysis.
The convergence of Algorithm \ref{alg:GRNBK} and \ref{alg:rGRNBK} are proved and their upper bound of the convergence rate given in Theorem \ref{thm:Convergence_tcc_linear_rate_grnbk}.

\begin{theorem}
   \rm	\label{thm:Convergence_tcc_linear_rate_grnbk}
	  Let $\eta>0$, $\hat x\in S$, $x_0\in B_{r,\varphi}(\hat x)$ and $\varphi$ be $\sigma$-strongly convex.
 Based on \cite[Assumption 1]{2024GLW} and the Jacobian matrix $F'(x)$ has full column rank for all $x\in B_{r,\varphi}(\hat x)$, it holds that 
	\begin{enumerate}
\item[(i)] If $\eta<\frac12$, then the iterates $x_k$ generated by  Algorithm \ref{alg:rGRNBK} fulfill that 
	\begin{equation} 
			\label{eqn:tcc_linear}
				\frac{\sigma}{2}\mathbb{E}[ \|x_k-\hat x\|_2^2 ]
				\leq \mathbb{E}[ D^{x_{k}^*}_\varphi(x_{k},\hat x) ]
				\leq \Big( 1- \frac{ 2\sigma\big(\frac12-\eta\big)}{n(1+\eta)^2M\kappa_{\min}^2} \Big)^k \mathbb{E}[ D^{x_{0}^*}_\varphi(x_0,\hat x)].
			\end{equation}			
 \item[(ii)] Let $\varphi$ be additionally $M$-smooth, $H_k\cap\text{dom } \varphi\neq\emptyset$ and $\eta < \frac{\sigma}{2M}$. Then, the iterates $x_k$ generated by Algorithm~\ref{alg:GRNBK} satisfy that 
			\begin{equation} 
				\label{eqn:tcc_linear_Alg2}
				\frac{\sigma}{2}\mathbb{E}[ \|x_k-\hat x\|_2^2] 
				\leq  \mathbb{E}[ D^{x_{k}^*}_\varphi(x_{k},\hat x) ]
				\leq \Big( 1- \frac{ 2\sigma\big(\frac12-\eta\frac{M}{\sigma}\big)}{n(1+\eta)^2M\kappa_{\min}^2} \Big)^k  \mathbb{E}[ D^{x_{0}^*}_\varphi(x_0,\hat x)].
			\end{equation}
		\end{enumerate} 
Here,
$\kappa_{\min} := \min\limits_{x\in B_{r,\varphi}(\hat x)} \min\limits_{\|y\|_2=1} \frac{\|F'(x)\|_F}{ \|F'(x)y\|_2}$ and $\|\cdot\|_F$ represents the Frobenius norm.
 \end{theorem}  
\begin{proof}  
From Lemma \ref{lem:Dphi_descent_estimate_tcc}, it holds that 
$$
D^{x_{k+1}^*}_\varphi(x_{k+1},\hat x) \leq D^{x_k^*}_\varphi(x_k,\hat x)- \tau \frac{\big(F_{i_k}(x_k)\big)^2}{\|\nabla F_{i_k}(x_k)\|_2^2}.
$$
By taking the full expectation on both sides of the above formula, we have
\begin{equation*} 
  \begin{aligned} 
  \mathbb{E}[D^{x_{k+1}^*}_\varphi(x_{k+1},\hat x)]& \leq \mathbb{E}[D^{x_k^*}_\varphi(x_k,\hat x)]- \sum_{i=1}^n\tau \frac{\big|F_{i_k}(x_k)\big|^2}{\| F(x_k)\|_2^2} \frac{\big(F_{i_k}(x_k)\big)^2}{\|\nabla F_{i_k}(x_k)\|_2^2}\\
  & \leq \mathbb{E}[D^{x_k^*}_\varphi(x_k,\hat x)]-  \frac{\sum_{i=1}^n \tau \big|F_{i_k}(x_k)\big|^4}{\| F(x_k)\|_2^2\sum_{i=1}^n\|\nabla F_{i_k}(x_k)\|_2^2}\\
  & \leq \mathbb{E}[D^{x_k^*}_\varphi(x_k,\hat x)]-  \frac{\sum_{i=1}^n \tau \big|F_{i_k}(x_k)\big|^2\sum_{i=1}^n  \big|F_{i_k}(x_k)\big|^2}{n\| F(x_k)\|_2^2 \| F'(x_k)\|_F^2}\\
  & \leq \mathbb{E}[D^{x_k^*}_\varphi(x_k,\hat x)]-  \frac{\sum_{i=1}^n \tau \big|F_{i_k}(x_k)-F_{i_k}(\hat x) \big|^2}{n \| F'(x_k)\|_F^2}.
\end{aligned}		
  \end{equation*}
From above inequality, the local tangential cone
condition \eqref{eqn:TCC2} and using \ref{lem:quot_diff_estimate}, we have
 
\begin{equation*} 
  \begin{aligned} 
\mathbb{E}[D^{x_{k+1}^*}_\varphi(x_{k+1},\hat x)] 
&\leq \mathbb{E}[D^{x_k^*}_\varphi(x_k,\hat x)] - \frac{\sum_{i=1}^n\frac{\tau}{(1+\eta)^2}|\langle\nabla F_{i_k}(x_k), x_k-\hat{x}\rangle|^2}{n\| F'(x_k)\|_F^2}\\
&\leq\mathbb{E}[D^{x_k^*}_\varphi(x_k,\hat x)] -\frac{\tau}{(1+\eta)^2} \frac{\| F'(x_k)(x_k-\hat{x})\|^2}{n\| F'(x_k)\|_F^2}\\  
&\leq \mathbb{E}[D^{x_k^*}_\varphi(x_k,\hat x)] -\frac{\tau}{n(1+\eta)^2\kappa^2_{\min}} \mathbb{E}[{\| x_k-\hat{x}\|^2}].
  \end{aligned}		
\end{equation*}
From $\varphi$ is $M$-smooth and \ref{lem:BasicsConvexAnalysis_convex_Lsmooth} (ii), we have
\begin{equation*}
\begin{aligned}		
    \mathbb{E}[D^{x_{k+1}^*}_\varphi(x_{k+1},\hat x)] &\leq \mathbb{E}[D^{x_k^*}_\varphi(x_k,\hat x)] -\frac{2\tau}{n(1+\eta)^2M\kappa^2_{\min}} \mathbb{E}[D^{x_k^*}_\varphi(x_k,\hat x)]\\
    &\leq \left(1-\frac{2\tau}{n(1+\eta)^2M\kappa^2_{\min}}\right) \mathbb{E}[D^{x_k^*}_\varphi(x_k,\hat x)]. 
\end{aligned}		
\end{equation*}
By induction and incorporating $\tau$ introduced in Lemma \ref{lem:Dphi_descent_estimate_tcc}, we complete the proof.
\end{proof}

\begin{remark}\rm
For $\varphi(x)=\frac{1}{2}\|x\|_2^2$, the part (i) of Theorem \ref{thm:Convergence_tcc_linear_rate_grnbk} recovers the result from \cite[Theorem 2]{2022WLB} as a special case. 
In Theorem \ref{thm:Convergence_tcc_linear_rate_grnbk}, unfortunately a more pessimistic rate obtained for Algorithm \ref{alg:GRNBK} compared with Algorithm \ref{alg:rGRNBK}, since the $\tau$ in (ii) is upper bounded by the $\tau$ in (i).
\end{remark}

\section{Numerical experiments} \label{secnumer_grnbk}
In this section, numerical experiments are presented to verify the efficiency of
the greedy randomized nonlinear Bregman-Kaczmarz method (abbreviated as `GRNBK') and its relaxed version (abbreviated as `rGRNBK') compared with nonlinear Bregman-Kaczmarz method (abbreviated as `NBK') and its relaxed version (abbreviated as `rNBK') proposed in \cite{2024GLW}.
 
\subsection{Sparse solutions of quadratic equations}\label{problem_LSQE}
Consider multinomial quadratic equations
\begin{equation}\label{eqmultiquadeq}
F_i(x) = \frac12\langle x,A^{(i)}x\rangle + \langle b^{(i)},x\rangle + c^{(i)}=0, 
\end{equation}
where $A^{(i)}\in\mathbb{R}^{d\times d}$, $b^{(i)}\in\mathbb{R}^d$, 
$c^{(i)}\in\mathbb{R}$ and $i=1,\cdots,n$. 
To find a sparse solution $\hat x\in\mathbb{R}^d$ of \eqref{eqmultiquadeq}, we adopt the nonsmooth distance generating function $\varphi(x)=\lambda\|x\|_1+\frac12\|x\|_2^2$ with sparsity parameter $\lambda>0$.
Since it holds $\text{dom }\varphi = \mathbb{R}^d$, it is always possible to choose the stepsize $t_{k,\varphi}$ from \eqref{eqn:BregProj_stepsize} in NBK and GRNBK method. Moreover, the stepsize can be computed exactly by a sorting procedure, as $\varphi^*$ is a continuous piecewise quadratic function, see \cite[Example 3.2]{2024GLW}. 

To guarantee existence of a sparse solution, we select a sparse vector $\hat x\in\mathbb{R}^d$, the data $A^{(i)}$ and $b^{(i)}$ generated with entries from the standard normal distribution and set 
$$
 c^{(i)} = - \Big( \frac{1}{2}\langle \hat{x},A^{(i)}\hat{x}\rangle + \langle b^{(i)},\hat{x}\rangle \Big). 
$$
In all iterations, the nonzero part of $\hat x$ and the initial subgradient $x_0^*$ are sampled from the standard normal distribution. The initial vector $x_0$ is computed by $x_0=\nabla\varphi^*(x_0^*) = S_\lambda(x_0^*)$.
The iterations terminate when the norm of residual $\left\|F(x_k) \right\|_2\leq 10^{-12}$ or the number of iterations exceeds 1000.
  
\begin{figure}[htb] 
\centering 
\vspace{-0.4cm}  
 \subfigtopskip=1pt  
\subfigbottomskip=0.1pt  
\subfigcapskip=-5pt 
	\subfigure[]
	{
		\begin{minipage}[t]{0.48\linewidth}
			\centering
			\includegraphics[width=1\textwidth]{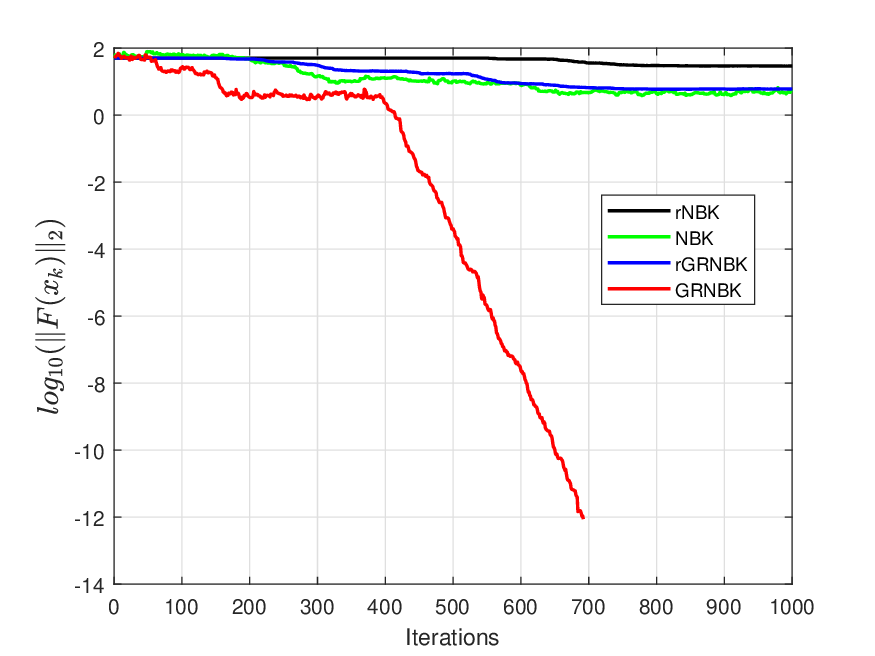}
		\end{minipage}
	}
	\subfigure[]
	{
		\begin{minipage}[t]{0.48\linewidth}
			\centering 
			\includegraphics[width=1\textwidth]{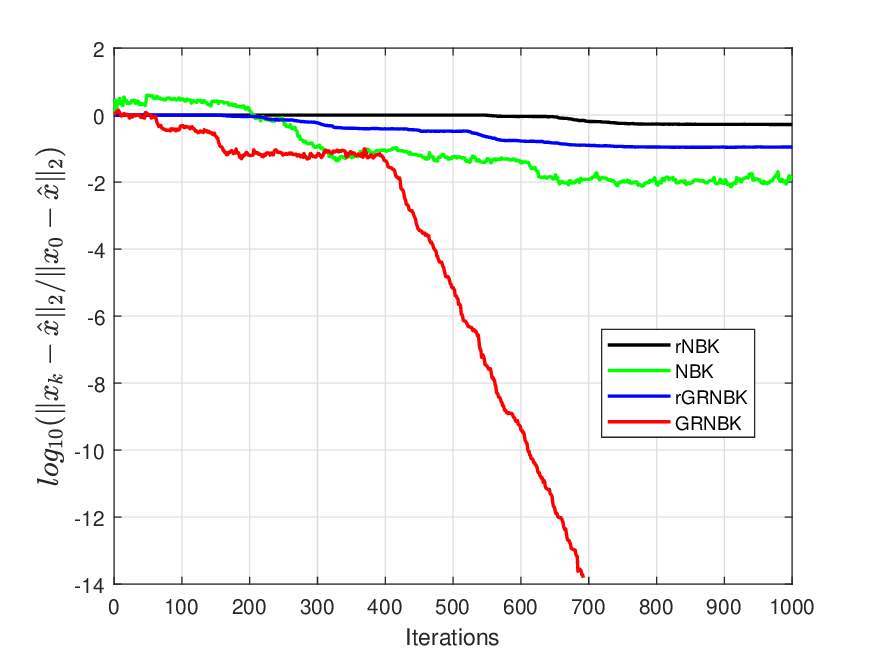}
		\end{minipage}
	}
  \caption{ Convergence curves of residual versus iterations for Problem \ref{problem_LSQE} with $(n,d)=(500, 100),s=10, \lambda=5$.} 
		\label{fig:reserrvsitcpuEx1ovlam5_grnbk}
\end{figure}

For a example, choosing $A^{(i)}\sim\mathcal{N}(0,1)^{500\times 500}$ for $i=1,...,100$, $\hat x$ with $10$ nonzero entries and $\lambda=5$. 
In Figure \ref{fig:reserrvsitcpuEx1ovlam5_grnbk}, the curves of the norm of the residual versus the number of iteration steps and the solution error versus the number of iteration steps are plotted, respectively.

From Figure \ref{fig:reserrvsitcpuEx1ovlam5_grnbk}, it is seen that the GRNBK method clearly converges faster than the NBK method and the rGRNBK method  converges faster than the rNBK method, which implies that the residual-based greedy strategies are effective and can greatly improve the convergence speed of nonlinear Bregman-Kaczmarz method.
It is seen from Figure \ref{fig:reserrvsitcpuEx1ovlam5_grnbk} that the convergence curves for the GRNBK method decreases much faster than those of its relaxed versions, which shows that the computation of the $t_{k,\varphi}$ step size for GRNBK method pay off.

\subsection{Linear systems on the probability simplex}\label{problem_LSPS}
Consider linear systems constrained to the probability simplex
\begin{equation}
	\label{eqn:linear_system_on_probability_simplex} 
	\text{find } x\in\Delta^{d-1}, \text{ s.t. } Ax=b.
\end{equation}
That is, choosing $F_i=\langle a_i,x\rangle - b_i$ and viewing $C=\Delta^{d-1}= \{x\in\mathbb{R}_{\geq 0}^d:\ \sum_{i=1}^d x_i=1\}$ as the additional constraint in problem \eqref{eqn:problem}. For NBK and GRNBK methods, the simplex-restricted negative entropy function from \cite[Example 3.3]{2024GLW} is used, which is given by
$$ 
\varphi(x) = \begin{cases}
\sum_{i=1}^d x_i \log(x_i), & x\in\Delta^{d-1}, \\
	+\infty, & \text{otherwise.}
\end{cases}  
$$
It is known from \cite[Example 3.3]{2024GLW} that $\varphi$ is $1$-strongly convex with respect to the $1$-norm $\|\cdot\|_1$. Therefore, rNBK and rGRNBK methods with $\sigma=1$ and $\|\cdot\|_*=\|\cdot\|_\infty$ are considered.   
Note that it holds $\nabla f_{i_k}(x)=a_{i_k}$ for all $x$ and $\beta_k = b_{i_k}$ in NBK and GRNBK method. 
If problem \eqref{eqn:linear_system_on_probability_simplex} has a solution, then condition \eqref{eqn:Condition_hyperplane_nonempty_intersection} is fulfilled in each step of NBK and GRNBK method, so these methods take always the stepsize $t_k=t_{k,\varphi}$ from the exact Bregman projection.  

Let the right-hand side be $b=A\hat x$ with the solution $\hat x$ drawn from the uniform distribution on the probability simplex $\Delta^{d-1}$. All iterations start from the center point $x_0=(\frac{1}{d}, ..., \frac{1}{d})$ and terminate when the norm of relative residual $\left\|Ax_k-b \right\|_2/\left\|Ax_0-b \right\|_2\leq 10^{-9}$ or the number of iterations exceeds 10,000.

In the first setting, the matrix $A$ is generated from standard normal distribution. Specifically, $A\sim\mathcal N(0,1)^{n\times d}$ with $(n,d) = (400,300)$ in the overdetermined case and $(n,d)=(300,400)$ in the underdetermined case, respectively. 

From Figure \ref{fig:resvsitEx2stdnorm}, it is seen that the GRNBK method converges faster than NBK method and the rGRNBK method converges faster than rNBK method both in the overdetermined and underdetermined case, which shows the  advantage of the greedy strategy. 

\begin{figure}[!htbp] 
\centering 
\vspace{-0.4cm} 
 \subfigtopskip=1pt  
\subfigbottomskip=0.1pt  
\subfigcapskip=-5pt 
	\subfigure[]
	{
		\begin{minipage}[t]{0.48\linewidth}
			\centering
			\includegraphics[width=1\textwidth]{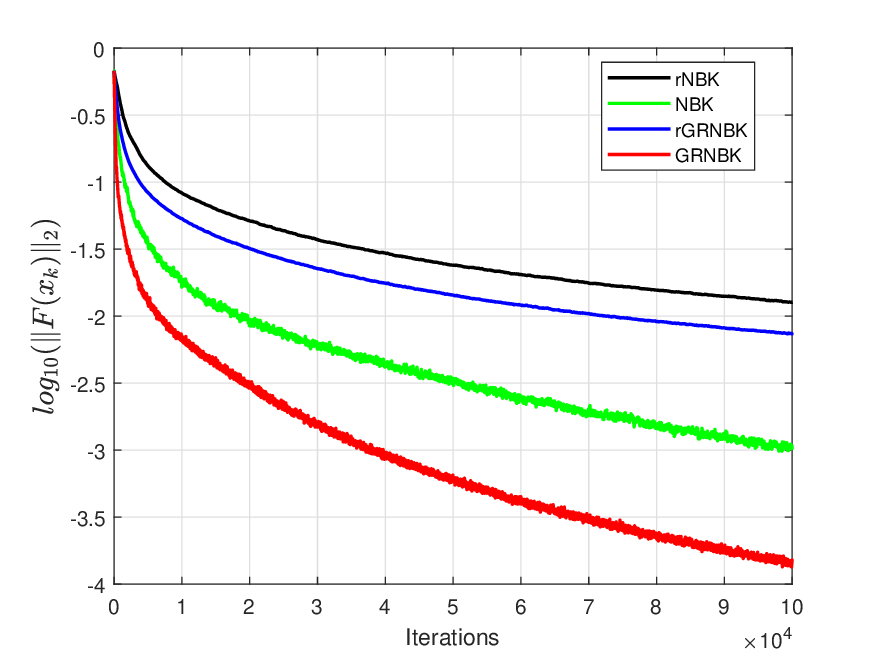}
		\end{minipage}
	}
	\subfigure[]
	{
		\begin{minipage}[t]{0.48\linewidth}
			\centering 
			\includegraphics[width=1\textwidth]{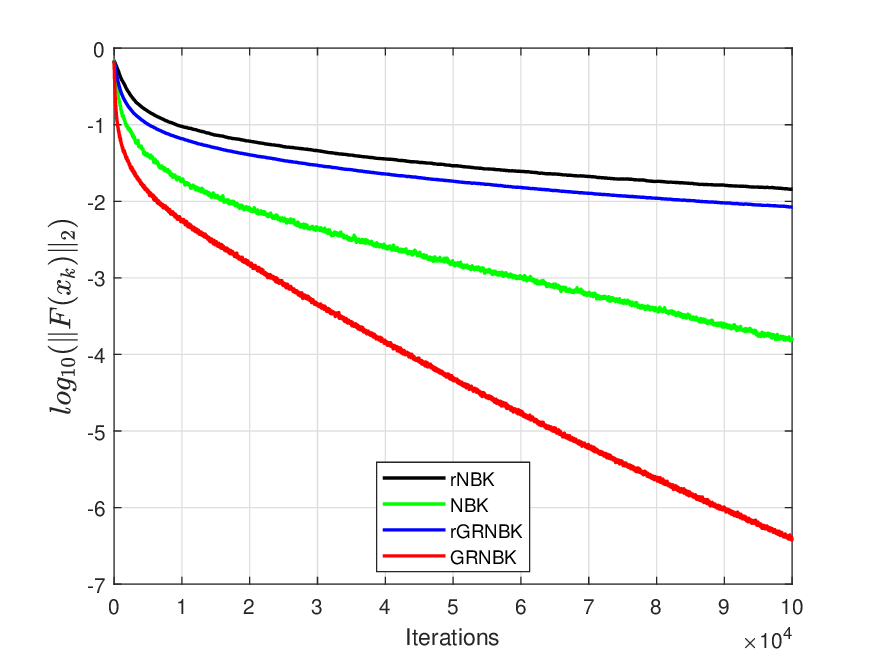}
		\end{minipage}
	}
  \caption{ Convergence curves of residual versus iterations for Problem \ref{problem_LSPS} with Left: $A\sim\mathcal{N}(0,1)^{400\times300}$ and Right: $A\sim\mathcal{N}(0,1)^{300\times400}$.} 
		\label{fig:resvsitEx2stdnorm}
\end{figure} 

In the second setting, to study the impact of the distribution of entries of $A$ on convergence quality of the methods, the tested matrix $A$ is generated from uniformly distributed entries $A\sim\mathcal{U}([0,1])^{n\times d}$ and $A\sim\mathcal{U}([0.9,1])^{n\times d}$ with $(n,d)=(300,400)$, respectively.  
From Figure \ref{fig:resvsitcpuEx2unif_grnbk}, it is seen that the
redundant rows of the matrix $A$ don not deteriorate the convergence of the GRNBK method. 
\begin{figure}[htb] 
\centering 
\vspace{-0.4cm} 
 \subfigtopskip=1pt  
\subfigbottomskip=0.1pt  
\subfigcapskip=-5pt 
	\subfigure[]
	{
		\begin{minipage}[t]{0.48\linewidth}
			\centering
			\includegraphics[width=1\textwidth]{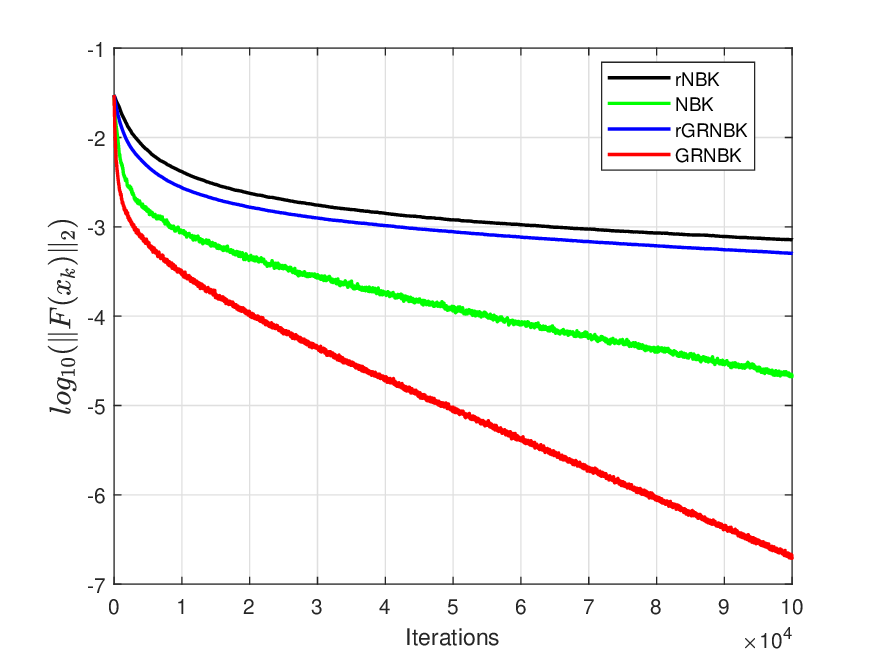}
		\end{minipage}
	}
	\subfigure[]
	{
		\begin{minipage}[t]{0.48\linewidth}
			\centering 
			\includegraphics[width=1\textwidth]{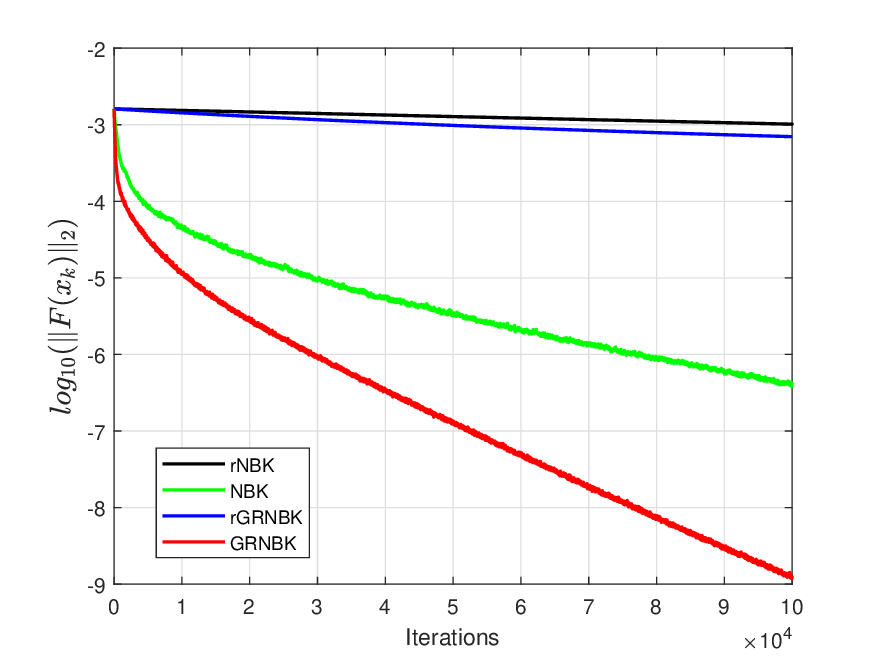}
		\end{minipage}
	} 
  \caption{ Convergence curves of residual versus iterations for Problem \ref{problem_LSPS} with Left: $A\sim\mathcal{U}([0,1])^{300\times400}$ and Right: $A\sim\mathcal{U}([0.9,1])^{300\times400}$.} 
	\label{fig:resvsitcpuEx2unif_grnbk}
\end{figure}

\subsection{Left stochastic decomposition}
\label{problem_LSD}
Consider the left stochastic decomposition problem formulated as follows:
\begin{equation}
	\label{eqn:LSD_problem}
	\text{find } X\in L^{r\times m},  \text{ s.t. } X^TX = A,
\end{equation}
where \[ L^{r,m}=\big\{P\in\mathbb{R}^{r\times m}_{\geq 0}: P^T \vmathbb{1}_r = \vmathbb{1}_m\} \] is the set of left stochastic matrices and $A\in\mathbb{R}^{r\times m}$ is a given nonnegative matrix. The problem is equivalent to the so-called {soft-K-means} problem and hence has applications in clustering \cite{2013AGKF}. Regarding \eqref{eqn:LSD_problem} as an instance of problem \eqref{eqn:problem} with component equations
\begin{equation*}
	F_{i,j}(X) = \langle X_{:,i}, X_{:,j} \rangle - A_{i,j} = 0, \qquad  i=1,...,r, \ j=1,...,m
\end{equation*}
and $C=L^{r\times m}\cong \left(\Delta^{r-1}\right)^m$, where $X_{:,i}$ represents the $i$th column of $X$.

For NBK and GRNBK method, the distance generating function is chosen from \cite[Example 3.4]{2024GLW} with the simplex-restricted negative entropy $\varphi_i=\varphi$ from see \cite[Example 3.3]{2024GLW}.
Since $F_{i,j}$ depends on at most two columns of $X$, NBK and GRNBK method act on $\Delta^{r-1}$ or $\Delta^{r-1}\times \Delta^{r-1}$ in each step. Therefore, the steps from \cite[Example 3.3]{2024GLW} is adopted in the first case, and from \cite[Example 3.5]{2024GLW} is used in the second case. 
Let $A=\hat X^T \hat X$ and the columns of $\hat X$ be sampled according to the uniform distribution on $\Delta^{r-1}$.
All iterations terminate when the norm of relative residual $\left\|F(X_k) \right\|_2\leq 10^{-5}$ or the number of iterations exceeds 300,000.

From Figure \ref{fig:resvsitcpuEx3_grnbk}, it is shown that the GRNBK method has the fastest convergence speed among all methods, which indicates the efficiency of the greedy strategy.

\begin{figure}[ htbp] 
\centering 
\vspace{-0.4cm} 
 \subfigtopskip=1pt  
\subfigbottomskip=0.1pt  
\subfigcapskip=-5pt  
	\subfigure[]
	{
		\begin{minipage}[t]{0.48\linewidth}
			\centering
			\includegraphics[width=1\textwidth]{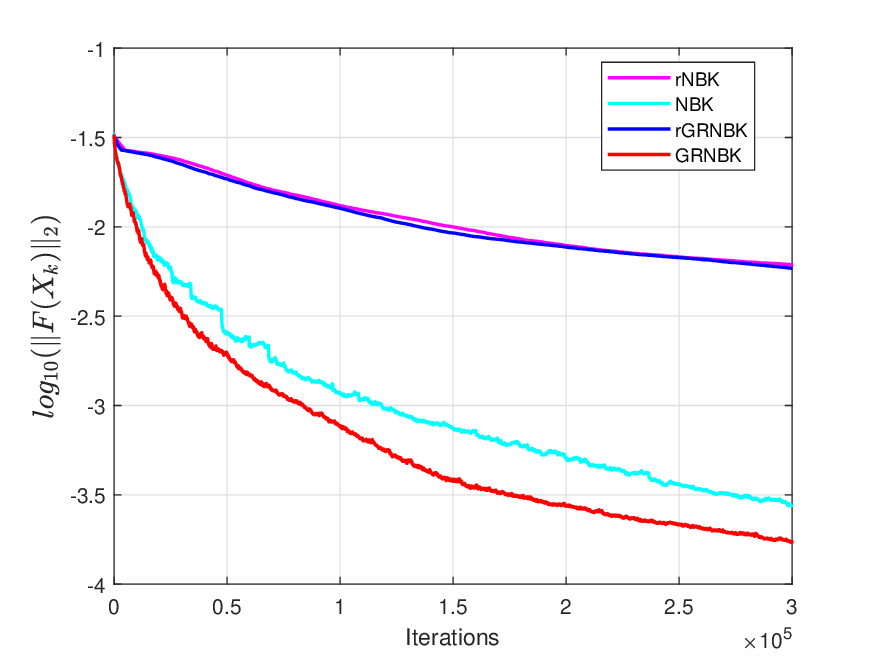}
		\end{minipage}
	}
	\subfigure[]
	{
		\begin{minipage}[t]{0.48\linewidth}
			\centering 
			\includegraphics[width=1\textwidth]{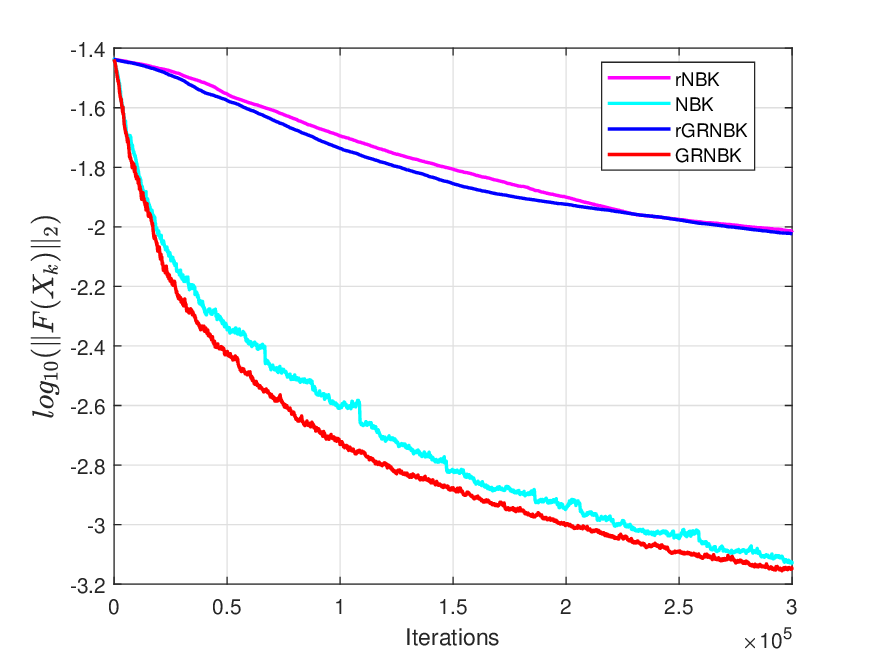}
		\end{minipage}
	} 
  \caption{ Convergence curves of residual versus iterations for Problem \ref{problem_LSD} with Left: $r=100, m=90$ and Right: $r=90, m=100$.} 	
		\label{fig:resvsitcpuEx3_grnbk}
\end{figure}

\section{Conclusions} \label{secconclu_grnbk} 
In this paper, a greedy randomized nonlinear Bregman-Kaczmarz method and its relaxed version are presented for the solution of constrained nonlinear systems. The convergence theory of the greedy randomized nonlinear Bregman-Kaczmarz method and its relaxed version are studied under the classical local tangential cone conditions. Numerical experiments by considering different constraints demonstrate the efficiency and robustness of the proposed method, which has a faster convergence speed than existing nonlinear Bregman-Kaczmarz methods.
In addition, some acceleration strategies for nonlinear Kaczmarz methods, such as averaging and sketching can be integrated into our methods. We leave these valuable topics to future work.\\

\noindent\textbf{Acknowledgements} Not Applicable.\\

\noindent\textbf{Author Contributions} All the authors contributed to the study conception, design and approval of the final manuscript.\\

\noindent\textbf{Funding} This work was supported by National Natural Science Foundation of China (No. 11971354).\\

\noindent\textbf{Availability of supporting data} No data was used for the research described in the article.

\section*{Declarations}
\noindent\textbf{Ethical Approval} Not Applicable.\\

\noindent\textbf{Competing Interests} The authors declare that they have no conflict of interest.\\

\bibliographystyle{plain}	 
\bibliography{refsGRNBK}

\begin{thebibliography}{10}

\bibitem{2013AGKF}
Raman Arora, Maya~R Gupta, Amol Kapila, and Maryam Fazel.
\newblock {Similarity-based Clustering by Left-Stochastic Matrix
  Factorization.}
\newblock {\em Journal of Machine Learning Research}, 14(7):1715--1746, 2013.

\bibitem{2021BW}
Zhong-Zhi Bai and Wen-Ting Wu.
\newblock {On greedy randomized augmented Kaczmarz method for solving large
  sparse inconsistent linear systems}.
\newblock {\em SIAM Journal on Scientific Computing}, 43(6):A3892--A3911, 2021.

\bibitem{2003BC}
Heinz~H Bauschke and Patrick~L Combettes.
\newblock {Iterating Bregman retractions}.
\newblock {\em SIAM Journal on Optimization}, 13(4):1159--1173, 2003.

\bibitem{2008DT}
Inderjit~S Dhillon and Joel~A Tropp.
\newblock {Matrix nearness problems with Bregman divergences}.
\newblock {\em SIAM Journal on Matrix Analysis and Applications},
  29(4):1120--1146, 2008.

\bibitem{2020FS}
Albert Fannjiang and Thomas Strohmer.
\newblock The numerics of phase retrieval.
\newblock {\em Acta Numerica}, 29:125--228, 2020.

\bibitem{2021GHT}
Guang-Yu Gao, Bo~Han, and Shan-Shan Tong.
\newblock {A projective two-point gradient Kaczmarz iteration for nonlinear
  ill-posed problems}.
\newblock {\em Inverse Problems}, 37(7):075007, 2021.

\bibitem{2024GLW}
Robert Gower, Dirk~A Lorenz, and Maximilian Winkler.
\newblock {A Bregman--Kaczmarz method for nonlinear systems of equations}.
\newblock {\em Computational Optimization and Applications}, 87(3):1059--1098,
  2024.

\bibitem{2021HM}
Jamie Haddock and Anna Ma.
\newblock {Greed works: An improved analysis of sampling Kaczmarz--Motzkin}.
\newblock {\em SIAM Journal on Mathematics of Data Science}, 3(1):342--368,
  2021.

\bibitem{2022JYN}
Benjamin Jarman, Yotam Yaniv, and Deanna Needell.
\newblock {Online signal recovery via heavy ball Kaczmarz}.
\newblock 1:276--280, 2022.

\bibitem{2022LZB}
Chao-Yue Liu, Li-Bin Zhu, and Mikhail Belkin.
\newblock {Loss landscapes and optimization in over-parameterized non-linear
  systems and neural networks}.
\newblock {\em Applied and Computational Harmonic Analysis}, 59:85--116, 2022.

\bibitem{2014LSW}
Dirk~A Lorenz, Frank Schopfer, and Stephan Wenger.
\newblock {The linearized Bregman method via split feasibility problems:
  analysis and generalizations}.
\newblock {\em SIAM Journal on Imaging Sciences}, 7(2):1237--1262, 2014.

\bibitem{2019FD}
Frank Sch{\"o}pfer and Dirk~A Lorenz.
\newblock {Linear convergence of the randomized sparse Kaczmarz method}.
\newblock {\em Mathematical Programming}, 173:509--536, 2019.

\bibitem{2022WLB}
Qi-Feng Wang, Wei-Guo Li, Wen-Di Bao, and Xing-Qi Gao.
\newblock {Nonlinear Kaczmarz algorithms and their convergence}.
\newblock {\em Journal of Computational and Applied Mathematics}, 399:113720,
  2022.

\bibitem{2024XYb}
A-Qin Xiao and Jun-Feng Yin.
\newblock {On averaging block Kaczmarz methods for solving nonlinear systems of
  equations}.
\newblock {\em Journal of Computational and Applied Mathematics}, 451:116041,
  2024.

\bibitem{2023XYZ}
A-Qin Xiao, Jun-Feng Yin, and Ning Zheng.
\newblock {On fast greedy block Kaczmarz methods for solving large consistent
  linear systems}.
\newblock {\em Computational and Applied Mathematics}, 42(3):119, 2023.

\bibitem{2022YLG}
Rui Yuan, Alessandro Lazaric, and Robert~M Gower.
\newblock {Sketched Newton--Raphson}.
\newblock {\em SIAM Journal on Optimization}, 32(3):1555--1583, 2022.

\bibitem{2011Y}
Ya-Xiang Yuan.
\newblock {Recent advances in numerical methods for nonlinear equations and
  nonlinear least squares}.
\newblock {\em Numerical algebra, control and optimization}, 1(1):15--34, 2011.

\bibitem{2023ZBLW}
Fei-Yu Zhang, Wen-Di Bao, Wei-Guo Li, and Qin Wang.
\newblock {On sampling Kaczmarz--Motzkin methods for solving large-scale
  nonlinear systems}.
\newblock {\em Computational and Applied Mathematics}, 42(3):126, 2023.

\bibitem{2023ZWZ}
Jian-Hua Zhang, Yu-Qing Wang, and Jing Zhao.
\newblock {On maximum residual nonlinear Kaczmarz-type algorithms for large
  nonlinear systems of equations}.
\newblock {\em Journal of Computational and Applied Mathematics}, 425:115065,
  2023.

\bibitem{2024ZLT}
Yan-Jun Zhang, Han-Yu Li, and Ling Tang.
\newblock {Greedy randomized sampling nonlinear Kaczmarz methods}.
\newblock {\em Calcolo}, 61(2):25, 2024.

\end{thebibliography}

\end{document}